\documentclass{amsart}
\usepackage{cite}
\usepackage{amssymb,amsmath,amsthm}

\newtheorem{theorem}{Theorem}
\newtheorem{teo}{Theorem}
\newtheorem{prop}[teo]{Proposition}
\newtheorem{cor}[teo]{Corollary}
\newtheorem{lema}[teo]{Lemma}

\theoremstyle{definition}

\theoremstyle{remark}

\def\Z{\mathbb{Z}}

\date{}
\def\be{\begin{enumerate}}
\def\ee{\end{enumerate}}

\def\Z{\mathbb{Z}}
\def\S{\mathbb{S}}

\def\Zel{\mathcal{Z}}

\def \kf{\hat{k}\left(n,t\right)}
\def \kfc{\hat{k}\left(0,t\right)}
\def \cls{C^l\left([0,2\pi]\right)}

\newcommand{\ust}[1]{\hat{u}^{*({#1})}}

\newcommand{\kst}[1]{\hat{k}^{*({#1})}}

\newcommand{\dif}[1]{\frac{d{#1}}{dt}}
\newcommand{\norma}[2]{\left\|{#1}\right\|_{#2}}

\newcommand{\zf}[1]{\hat{z}({#1},t)}

\title{On the rate of convergence of the $p$-curve shortening flow}
\author{Jean C. Cortissoz, Andr\'es Galindo and Alexander Murcia}
\address{Department of Mathematics, Universidad de los Andes, Bogot\'a DC, Colombia.}
\begin{document}

\maketitle
\begin{abstract}
    In this paper we give rates of convergence for the $p$-curve shortening flow
    for $p\geq 1$ an integer, which improves on the known estimates and which are
    probably sharp.
\end{abstract}
\section{Introduction}

Let us first introduce the main character of this story, the $p$-curve
shortening flow, with $p$ a positive integer. So, we let
\[
x:\,\mathbb{S}^1\times\left[0,T\right)\longrightarrow \mathbb{R}^2
\]
be a family of smooth convex embeddings of $\mathbb{S}^1$, the unit circle, into $\mathbb{R}^2$.
We say that $x$ satisfies the $p$-curve shortening flow, $p\geq 1$, if
$x$ satisfies
\begin{equation}
\label{generalflow}
\frac{\partial x}{\partial t}=-\frac{1}{p}k^{p}N,
\end{equation}
where $k$ is the curvature of the embedding and $N$ is the normal vector pointing outwards the
region bounded by $x\left(\cdot,t\right)$. 

This is just a natural generalisation of the well known and well
studied curve shortening flow. A solution to \ref{generalflow} starting from an embedded convex simple curve
will contract, via embedded convex curves, towards a round point in finite time: this means
that if we start with a simple convex curve, via the $p$-curve shortening, after
a convenient normalisation, which includes a time reparametrisation, the embedded curves converge smoothly to a circle
(see \cite{Andrews}). It is also known that this convergence
is exponential in the following sense (here $\tilde{k}$ denotes the curvature of the embedded curves after
normalisation)
\[
\left\|\tilde{k}^{\left(n\right)}\right\|_{\infty}\leq Ce^{-\delta t}, 
\]
with where $\tilde{k}^{\left(n\right)}$ represents the $n$-derivative of 
$\tilde{k}$ with respect to the arclength parameter in $\mathbb{S}^1$, $n\geq 1$, and $\delta>0$. For the curve shortening flow, we can use
as $\delta$ any $2\alpha$ for $0<\alpha<1$, this was proved
by Gage and Hamilton in their by now famous (by mathematical standards)
paper \cite{GageHamilton}. For the $p$-curve shortening, Huang in \cite{Huang} showed that 
$\delta$ can be taken as $2\alpha p$, with the same restrictions on 
$\alpha$. Interestingly enough, with
the exception of the curve shortening 
flow ($p=1$), it has not been showed that
$\tilde{k}\rightarrow 1$ exponentially!  
For the curve shortening flow ($p=1$),
in the book \cite{ChouZhu} exponential convergence
of the curvature towards 1 is shown, and
Andrews and Bryan 
showed in \cite{AndrewsBryan} (although they did not stated explicitly) that
$\tilde{k}\rightarrow 1$ as fast as $e^{-2\tau}$.

Related to this problem is the mean curvature flow, and Sesum in \cite{Sesum},
using Huisken's work as a departure point, has given
sharp rates of convergence for this flow. 

The main goal of this paper to give better rates of convergence 
for the $p$-curve shortening flow, $p\geq 1$ an integer, than the ones previously known. 
Our main result, from which the said rates of convergence can be deduced,
is the following.

\begin{theorem}
\label{mainresult}
	Let $\psi>0$ be the curvature of the initial condition to (\ref{generalflow}). Then there exists a constant $c_p>0$ such that if
		
		\begin{equation}
		\label{smallnesscondition}
		\hat{\psi}(0)\geq c_p\norma{\psi}{2},
		\end{equation} 
		
		then the solution to the normalised $p$-curve shortening flow ($p$ a positive integer), that is
		 for the curvature $\tilde{k}$ of the curves given by the rescaled embedding 
		 $\left(\dfrac{p+1}{p}\right)^{\frac{1}{p+1}}\left(T-t\right)^{-\frac{1}{p+1}}x$, with rescaled 
		 time parameter $\tau=-\dfrac{1}{p+1}\log\left(1-\dfrac{t}{T}\right)$,
		 where $0<T<\infty$ is the maximum time of existence for (\ref{generalflow}),
		it holds that
\[
\left\|\tilde{k}-1\right\|_{C^{l}\left(\mathbb{S}^1\right)}\leq C_{p,l}e^{-\left(3p-1\right)\tau},
\]
where $C_{p,l}$ is a constant that only depends on $p,l$ and $\psi$.
\end{theorem}
Together with Theorem I1.1 from \cite{Andrews}, this gives the following
\begin{theorem}
For any simple convex curve as initial data, the normalised version of (\ref{generalflow}), converges towards a circle smoothly and the
curvature of the normalised embeddings satisfy
\[
\left\|\tilde{k}-1\right\|_{C^{k}\left(\mathbb{S}^1\right)}\leq C_{p,k}e^{-\left(3p-1\right)\tau},
\]
where $C_{p,k}$ is a constant that only depends on $p,k$ and the curvature of the initial condition.
\end{theorem}
Indeed, by the theorem of Andrews referred to above, (\ref{smallnesscondition}) eventually holds if we start (\ref{generalflow}) with
a given convex simple curve as initial data. The rate of convergence given by our main result seems to be the sharpest possible rate
of convergence for the $p$-curve shortening flow (see the remark at the end of this paper).

A naive idea for proving Theorem \ref{mainresult} would
be to use the Parabolic PDE to which the normalised version of $p$-curve shortening
flow is equivalent to (see equation (\ref{eq:bvpnormalized}) in Section \ref{wavenumberscontrolled}), and then linearise around the
steady solution to obtain exponential convergence.  However, if we linearise 
around the steady solution, the elliptic part of the parabolic operator corresponding to the $p$-curveshortening flow
has a negative eigenvalue, so no exponential convergence
should be expected (see the discussion in \cite{CortissozMurcia} right after Theorem 2.2, and 
notice that when $\lambda=1$, a negative eigenvalue occurs). The good news here is that, being $k$ a curvature, it satisfies
an important identity, which is responsible for us being
able to get this exponential convergence. 

Our methods
are based on the techniques employed in \cite{CortissozMurcia}, that is to say on the Fourier method.
Hence, we will transform our problem into (finite dimensional) approximations  
of an infinite dimensional dynamical system, for which appropriate estimates will be proved,
and which will finally lead to a proof of Theorem \ref{mainresult}, proof which is given in the final section of this paper.
The intermediate sections are devoted to show these appropriate estimates, which, 
in short, amount to controlling the Fourier wavenumbers of a solution to (\ref{generalflow}) in terms
of the average of the curvature; from this we will be able to show a time decay for the Fourier wavenumbers different
from the average (which in fact blows-up), and which, as we said before, will lead to a proof of Theorem \ref{mainresult}.

\section{Basic definitions and notation}
\label{Basicdefinitions}

When the initial curve is convex, the $p$-curve shortening
flow is equivalent to the following
 Boundary Value Problem:
\begin{equation}
\begin{cases} 
\dfrac{\partial k}{\partial t}=k^2\left(k^{p-1}\dfrac{\partial^2k}{\partial \theta^2}+(p-1)k^{p-2}\left(\dfrac{\partial k}{\partial \theta}\right)^2+\dfrac{1}{p}k^{p}\right)\quad\text{in}\quad \left[0,2\pi\right]\times\left(0,T\right) \\
k(\theta,0)=\psi(\theta)\quad\text{on}\quad\left[0,2\pi\right],\\ 
\end{cases}
\label{eq:bvp} 
\end{equation}

$p\in\Z^+$, with periodic boundary conditions, and $\psi$ a \emph{strictly positive function.} 
Notice that the Maximum Principle implies that $k$ must remain positive for all times (i.e. a convex
curve remains convex).
We will need to compute 
finite dimensional approximations of the previous partial
differential equation in Fourier space,
so we must establish some definitions and notation. 
Recall that
For $f\in L^2\left[0,2\pi\right]$, its Fourier expansion is given by:

\[\sum_{n\in\Z}\hat{f}\left(n\right)e^{i n\theta}\]

where, 

\[\hat{f}(n)=\frac{1}{2\pi}\int_{0}^{2\pi} f(\theta) e^{-in\theta }\,d\theta.\]
We shall refer to $\hat{f}\left(n\right)$ as the \emph{Fourier wavenumbers} of $f$.

We will also adopt the notation 
\[\displaystyle{\ust{m}(q_1,q_2,\dots,q_{m},t)=\hat{u}(q_1,t)\hat{u}(q_2,t)\cdots\hat{u}(q_m,t)},\]

\[H(p,q_1,q_2)=\frac{1}{p}-(p-1)q_1 q_2 -q_1^2,\]
and define the following sets
\[\mathcal{B}_{n}=\left\{\left(q_1,\dots,q_{p+2}\right)\in \Z^{p+2}:q_{p+2}=n-q_1-\cdots-q_{p+1}\right\},\]

\[\mathcal{A}_{n}=\left\{\mathbf{q}\in\mathcal{B}_{n}:\text{there are } 1\leq i < j\leq p+2\text{ such that } b_i\neq 0\text{ and } b_j\neq 0\right\},\]

and, 

\[\mathcal{C}_n=\left\{\mathbf{q}\in\mathcal{A}_{n}: q_j\neq 0, \pm 1,\text{for all } 1\leq j \leq p+2\right\}. \]
From now on $\mathcal{Z}$ will denote a finite set of integers which contains 0 (i.e, $0\in\mathcal{Z}$),
and which is symmetric around 0 (i.e., if $n\in\mathcal{Z}$ then $-n\in\mathcal{Z}$).

Using this notation, in Fourier space, the $p$-curve shortening
flow can be approximated by the following finite dimensional
dynamical system:

\begin{equation}
\begin{cases} 
 \dif{}\hat{k}(0,t)&=\dfrac{1}{p}\kfc^{p+2}+\sum_{\mathbf{q}\in\mathcal{A}_0\cap \Zel^{p+2}}H\left(p,q_1,q_2\right)\kst{p+2}\left(\mathbf{q},t\right),\\
 \\
\dif{}\hat{k}(n,t)&=\left(\dfrac{p+2}{p}-n^2\right)\kfc^{p+1}\kf\\&+\sum_{\mathbf{q}\in\mathcal{A}_n\cap \Zel^{p+2}}H\left(p,q_1,q_2\right)\kst{p+2}\left(\mathbf{q},t\right),\text{ if } n\neq 0,n\in\Zel\\
\end{cases}
\label{eq:efourier}
\end{equation}

with initial condition 

\[\hat{k}(n,0)=\hat{\psi}(n),\text{ if }n\in\Zel.\]

Formally, the $\hat{k}$ in the system right above should bear, for instance, a subindex which makes its dependence on $\mathcal{Z}$
explicit, but as this is understood
from now on, we will suppress it in what follows (and as our estimates will not depend on $\mathcal{Z}$, this
should be of no importance).

Notice also that (\ref{eq:efourier}) is an autonomous system, 
so there is a unique and smooth solution for a short time (see \cite{CoddingtonLevinson}). We will also make use of the seminorms $\norma{\cdot}{\beta}$,
which are defined as in  \cite{CortissozMurcia} as follows:

\[\norma{f}{\beta}=\max\left\{\sup_{\xi\in\Z}\left|\xi\right|^{\beta}\left|Re\left(\hat{f}\left(\xi\right)\right)\right|,\sup_{\xi\in\Z}\left|\xi\right|^{\beta}\left|Im\left(\hat{f}\left(\xi\right)\right)\right|\right\}.\]

As usual, we define $\cls$, $l=0,1,2,\dots,$ as the space of functions with continuous derivatives of order $l$, equipped with the norm

\[\norma{f}{\cls}=\max_{j=0,\dots,l}\sup_{\theta\in[0,2\pi]}\left|\frac{d^j f(\theta)}{d\theta^j}\right|\] 

\section{Technical Lemmas and intermediate results}

We shall follow closely the arguments presented in \cite{Murcia}. Therefore we must show that for given a solution to (\ref{eq:efourier}), we can
control the Fourier wavenumbers different from $0$ in terms of the $0$-th wavenumber. The fact that the first eigenvalue $\lambda_1<0$ is the main difficulty
we must face, as this makes difficult to control the $\pm 1$- wave number in terms of the $0$ th wave number. 
Once we have done this, all that is left is to follow the arguments presented in \cite{CortissozMurcia, Murcia}. The key to our proofs is that a curvature function of a locally convex curve satisfies 

\[Q(k)=\int_{0}^{2\pi}\frac{e^{i\theta}}{k(\theta,t)}\,d\theta=0,\] 

since this identity, once we have control over the higher Fourier wavenumbers (those with $\left|n\right|\geq 2$) assuming
control over the $\pm 1$ wavenumbers, allows us to control the $\pm 1$ Fourier wavenumbers.

The careful reader must notice that the proofs given in this paper, our estimates are given for
system (\ref{eq:efourier}), and that this estimates are independent of $\mathcal{Z}$, this
allows us to take a limit so the results are valid for the full system (\ref{eq:bvp}).

\subsection{Controlling the Fourier wavenumbers} 

We start with a technical lemma. 
\begin{lema}
	There is a  $0<\delta <\frac{1}{4}$ such that if the initial condition $\psi$ of (\ref{eq:bvp}) satisfies:
	
	\[2\delta\cdot\psi(0)\geq q^2|\hat{k}(q,t)|,\]
	
	and 
	
	\[\hat{k}(0,t)\geq\left(1-\delta\right)\hat{\psi}(0),\]
	for $t\in (0,\tau)$, then
	$\hat{k}\left(0,t\right)$
	is non decreasing.\label{l:tec}
\end{lema}
\begin{proof}
   From the hypothesis of the lemma,

	\[
	U:=\sum_{\mathbf{q}\in\mathcal{A}_0\cap \Zel^{p+2}}H\left(p,q_1,q_2\right)\kst{p+2}\left(\mathbf{q},t\right)=O\left(\delta\hat{k}(0,t)^{p+2}\right),
	\]
	and the implicit constnat in the big $O$ notation does not depend on $\mathcal{Z}$.
	Hence, for $\delta>0$ small enough, 
	the term $\dfrac{1}{p}\hat{k}\left(0,t\right)$ dominates the the term $U$ in differential equation for $\hat{k}\left(0,t\right)$. 
	This implies that $\dfrac{d}{dt}\hat{k}\left(0,t\right)>0$, and the
	conclusion of the lemma follows.
	
\end{proof}

Now we show some control estimates for the Fourier modes,

\begin{lema}
	There is a  $0<\delta <\frac{1}{4}$ such that if the initial condition $\psi$ of (\ref{eq:bvp}) satisfies:
	\[2\delta\cdot\hat{\psi}(0)\geq |\hat{\psi}(\pm 1)|\quad\text{and}\quad \delta\cdot\hat{\psi}(0)\geq q^2 |\hat{\psi}(q)|\quad\text{for}\quad \left|q\right|\geq 2\]
	
	holds, and for $t\in (0,\tau)$
	
	\[2\delta\cdot\hat{\psi}(0)\geq|\hat{k}(\pm 1,t)|,\]
	
	and
	
	\[\hat{k}(0,t)\geq \left(1-\delta\right)\hat{\psi}(0).\]
	
	Then
	\[\delta\cdot\hat{\psi}(0)\geq q^2|\hat{k}(q,t)|.\]
	\label{l:c1}
\end{lema}
\begin{proof}Let us consider the quantity $R_n=\left|\dfrac{\kf}{\hat{\psi}(0)}\right|,$ and we prove that is nonincreasing for $n$ fixed. 
	
	We compute:
	
	\[\dif{}\log R_n=\left(\frac{p+2}{p}-n^2\right)\kfc^{p+1}+\sum_{i=1}^{2}B_i\]
	
	where the $B_i$ terms are given by:
	
	\begin{align*}
	B_1&=\frac{1}{\kf}\sum_{\mathbf{q}\in\mathcal{C}_n\cap \Zel^{p+2}}H\left(p,q_1,q_2\right)\kst{p+2}\left(\mathbf{q},t\right)\\
	B_2&=\frac{1}{\kf}\sum_{\mathbf{q}\in\mathcal{A}_n\setminus \mathcal{C}_n\cap \Zel^{p+2}}H\left(p,q_1,q_2\right)\kst{p+2}\left(\mathbf{q},t\right)\\
	\end{align*} 
	
	We bound $B_1$, 
	
	\begin{eqnarray*}
	\left|\kf\right|\left|B_1\right|&\leq&\frac{1}{p}\sum_{\mathbf{q}\in\mathcal{C}_n\cap \Zel^{p+2}}\left|\kst{p+2}\left(\mathbf{q},t\right)\right|\\
	&&+(p-1)\sum_{\mathbf{q}\in\mathcal{C}_n\cap \Zel^{p+2}}\left|q_1\right| \left|q_2\right|\left|\kst{p+2}\left(\mathbf{q},t\right)\right| \\
	&&+\sum_{\mathbf{q}\in\mathcal{C}_n\cap \Zel^{p+2}}q_1^2\left|\kst{p+2}\left(\mathbf{q},t\right)\right|
	\end{eqnarray*}
	Now if $\mathbf{q}\in\mathcal{C}_n\cap \Zel^{p+2}$, then 
	
	\[\left|\kst{p+2}\left(\mathbf{q},t\right)\right|=\frac{2^{p+2}\delta^{p+2}\hat{\psi}(0)^{p+2}}{q_1^2\cdots q_{p+1}^2\left(n-q_1-\cdots-q_{p+1}\right)^2},\]
	and hence
	
	\[\left|\kf\right|\left|B_1\right|\leq\delta^{p+2} C'_{p}\hat{\psi}(0)^{p+2},\]
	with $C'_p$ independent of $\mathcal{Z}$.
	Since $|\kf|=|\kfc|/n^2$, we get 
	
	\begin{equation}
	\left|B_1\right|\leq \delta^{p+2}C'_{p}n^2\hat{\psi}(0)^{p+1}.
	\label{eq:1estim}
	\end{equation}
	
	Splitting the sums and using similar calculations as in (\ref{eq:1estim}), we obtain
	
	\begin{equation}
	\left|B_2\right|\leq \delta C''_{p}n^2\hat{\psi}(0)^{p+1},
	\label{eq:3-4estim}
	\end{equation}
	and again $C''_{p}$ is independent of $\mathcal{Z}$.
	
	Since the sum of the absolute value of all these terms can be made smaller than $$\left(n^2 -\dfrac{p+2}{p}\right)\hat{k}(0,t)^{p+1},$$
	for $n\geq 2$
	by taking $\delta>0$ small enough, then the $R_m$ term is non increasing for $\delta>0$ small enough. From this
	the conclusion of the lemma follows.
	\end{proof}

As we have been doing so far, in what follows the wave numbers are restricted to a fixed but arbitrary set $\mathcal{Z}$, 
so keep this in mind.
And, as announced at be beginning of this section, since the estimates are independent of
$\mathcal{Z}$, a limiting procedure will give the result for when we take as $\mathcal{Z}$ the whole set of integers.

\begin{lema}
\label{controllingsmallwavenumbers}
	Let $\psi$ be such that $Q\left(\psi\right)=0$. There is a $0<\delta'<\frac{1}{4}$ such that if $0<\delta\leq \delta'$ 
	and $\delta\hat{\psi}\left(0\right)\geq \left|\hat{\psi}\left(\pm 1\right)\right|$ 
	then whenever $\left|\kf\right|n^2\leq \delta \kfc$ for all $\left|n\right|\geq 2$, for all $t\in\left[0,\tau\right]$,
	then we also have  $\left|\hat{k}\left(q,t\right)\right|q^2\leq \delta \kfc$ for $q=\pm 1$ on the same time interval.
	\label{le:acote}
\end{lema}

\begin{proof}
Since we have that  $\delta\hat{\psi}\left(0\right)\geq \left|\hat{\psi}\left(\pm 1\right)\right|$, we can choose a $\tau'\in \left[0,\tau\right]$
such that $\left|\hat{k}\left(\pm 1,t\right)\right|\leq 2\delta \kfc$ on $\left[0,\tau'\right]$
(remember we are working with an arbitrary but final dimensional approximation of the $p$-curve shortening flow).
	We have the following identity
	
		\[\frac{1}{k\left(\theta,t\right)}=\frac{1}{\kfc}\frac{1}{1+\sum_{q\neq 0}\frac{\hat{k}(q,t)}{\kfc}e^{iq\theta}}=\frac{1}{\kfc}\frac{1}{1+z}=\frac{1}{\kfc}\sum_{n=0}^{\infty}(-1)^n z^n,\]
	where $z=z(\theta,t)=\sum_{q\neq 0}\frac{\hat{k}(q,t)}{\kfc}e^{iq\theta}$. It can be easily seen that for the Fourier modes of $z$ we have:
	
	\[\displaystyle{\zf{p}=\begin{cases}
	0&\text{ if } p=0\\
	\frac{\hat{k}(p,t)}{\kfc}&\text{ otherwise}
	\end{cases}},\] 

    Now taking the Fourier transform, this implies
    
    \[\widehat{\left(\frac{1}{k}\right)}(-1,t)=\frac{1}{\kfc}\left(-\zf{-1}+\sum_{m=2}^{\infty}(-1)^m\sum_{q_1+\cdots+q_m=-1} \zf{q_1}\cdots\zf{q_m}\right)\]
    
   Since $k$ is the 
    curvature of a convex curve, we have $Q\left(k\right)=0$, so
    
    \[\widehat{\left(\frac{1}{k}\right)}(-1,t)=\int_{\S^1}\frac{e^{-(-1)i\theta}}{k\left(\theta,t\right)}\,d\theta=Q(k)=0,\]
    
    then 
    
    \[\left|\zf{-1}\right|=\left|\sum_{m=2}^{\infty}(-1)^m\sum_{q_1+\cdots+q_m=-1} \zf{q_1}\cdots\zf{q_m}\right|\]
    
    In order to estimate the sum in the right side, let us notice that
    
\begin{align*}
    \sum_{q_1+\cdots+q_m=-1} \zf{q_1}\cdots\zf{q_m}&=\sum_{j=0}^{m-1}\binom{m}{j}\zf{1}^j\sum_{l=0}^{m-j-1}\binom{m-j}{l}\zf{-1}^j\\
                          &\sum_{q_{j+l+1}+\cdots+q_m=-1-j+l} \zf{q_{j+l+1}}\cdots\zf{q_m}\\
    \end{align*}
    
    Now we proceed to estimate $\displaystyle{S_m=\sum_{q_{1}+\cdots+q_m=-1} \zf{q_{1}}\cdots\zf{q_m}}$,
    
    \begin{align*}
    \left|S_m\right|&\leq\sum_{j=0}^{m-1}\binom{m}{j}\left|\zf{1}\right|^j\sum_{l=0}^{m-j-1}\binom{m-j}{l}\left|\zf{-1}\right|^l\sum_{\substack{q_{j+l+1}+\cdots+q_m\\=-1-j+l}} \delta^{m-j-l}\frac{1}{q_{j+l+1}^2}\cdots\frac{1}{q_{m}^2}\\
    &\leq\sum_{j=0}^{m}\binom{m}{j}\left|\zf{1}\right|^j\sum_{l=0}^{m-j}\binom{m-j}{l}\left|\zf{-1}\right|^l\delta^{m-j-l}C_1^{m-j-l}\\
    &=\left(\left|\zf{1}\right|+\left|\zf{-1}\right|+\delta C_1\right)^m,
    \end{align*}
    where $C_1$ is a constant independent of $\mathcal{Z}$ and $\delta$.
    
    As we have that  $\left|\zf{\pm1}\right|\leq 2\delta$,  we get 
    
    \[\left|S_m\right|\leq \left(\left|\zf{1}\right|+\left|\zf{-1}\right|+\delta C_1\right)^m\leq \delta^m\left(4+C_1\right)^m.\]
    
    Therefore 
    
    \[\left|\zf{-1}\right|\leq\frac{\delta^2(4+C_1)^2}{1-\delta(4+C_1)}\leq \delta\]
    
    as long as $\delta\leq \frac{1}{2\left(4+C_1\right)}$.    
\end{proof}

\begin{lema}
There is a $0<\delta<\frac{1}{4}$, independent of $\mathcal{Z}$, such that if the initial condition $\psi$ of (\ref{eq:bvp}) satisfies $Q(\psi)=0$ and 

\[\delta\cdot\hat{\psi}(0)\geq q^2|\hat{\psi}(q)|,\]

then for all times $t$ (as long as the solution to (\ref{eq:efourier}) exists),

\[\delta\psi(0)\geq q^2\left|\hat{k}(q,t)\right|.\]

\end{lema}

\begin{proof}
There exists a $\tau>0$ be such that the interval $[0,\tau]$ is maximal with respect to the following property: for all $t\in[0,\tau]$ we have

\[2\delta\cdot\hat{\psi}(0)\geq q^2|\hat{k}(q,t)|,\]

and

\[\hat{k}(0,t)\geq (1-\delta)\hat{\psi}(0).\]

Now applying Lemmas \ref{l:tec} and \ref{l:c1}, we obtain that

\[\delta\psi(0)\geq q^2\left|\hat{k}(q,t)\right|,\]

whenever $\left|q\right|\geq 2$. Now Applying Lemma \ref{le:acote}, we get that

\[\delta\psi(0)\geq\left|\hat{k}(\pm 1,t)\right|, \]

when $t\in [0,\tau]$. Hence we have that  $\delta\psi(0)\geq q^2\left|\hat{k}(q,\tau)\right|$ and if we apply the same arguments as before we can show that there is a $\tau_1>\tau$ such that if $t\in\left[0,\tau_1\right]$ then $\delta\cdot\hat{\psi}(0)\geq q^2\left|\hat{k}(q,t)\right|$, contradicting the maximality of $\left[0,\tau\right]$.
\end{proof}

For $\delta>0$ small enough, assuming that for the initial condition we have
$\delta \hat{\psi}\left(0\right)\geq q^2|\hat{\psi}(q)|$, we have now control over all the
Fourier wavenumbers of the solution. The arguments in \cite{CortissozMurcia} now apply almost verbatim: see the upcoming sections.

\subsection{Decay of the Fourier wavenumbers}
\label{wavenumberscontrolled}

Again, all the estimates proved in this section are valid for any choice of $\mathcal{Z}$, and are also independent of the choice. 
Our main purpose is to show that the Fourier wavenumbers $\hat{k}\left(n,t\right)$, $n\neq 0$, go to 0 as
$t\rightarrow T$.
To begin, we have, as in \cite{CortissozMurcia}, the Trapping Lemma (Lemma 3.2 in \cite{CortissozMurcia}).
Keep in mind that we are always under the assumption that $\psi>0$ is the curvature function of a simple convex closed curve (or equivalently, the 
identity $Q\left(\psi\right)=0$ holds).

\begin{theorem}
[Trapping Lemma]
There exists a constant $c_p>0$ independent of the choice of $\Zel$ such that if the initial datum $\psi$ satisfies the following inequality:

\[\hat{\psi}(0)\geq c_p\norma{\psi}{2},\]

then there exists a $\gamma>0$ that depends on $\psi$ such that the solution to (\ref{eq:efourier}) satisifies:

\begin{equation}
\left|\kf\right|\leq \frac{\hat{\psi}(0)e^{-\gamma|n|t}}{c_p|n|^2},\quad n\neq 0.
\end{equation}
\end{theorem}

Also, in the same way as Lemmas 3.3 and 3.4 are obtained in \cite{CortissozMurcia}, we have a 
Blow-up Lemma.
\begin{lema}[Blow-up] 
There is a $c_p>0$ (the same as in the Trapping Lemma) such that if the initial condition $\psi$ of (\ref{eq:bvp}) satisfies 
\[\cdot\hat{\psi}(0)\geq c_p\left\|\hat{\psi}\right\|_2,\]
then there are constants a number $c,c'>0$ such that:

\begin{equation}
\frac{c}{T-t}\leq \kfc^{p+1}\leq \frac{c'}{T-t}.
\label{eq:blowup1}
\end{equation}

\end{lema}
From now on, we assume that $\psi$ satisfies
\[\hat{\psi}(0)\geq c_p\norma{\psi}{2},\]
where $c_p$ is such that the Trapping Lemma holds.

We also have a few important observations. First, integrating the ODE for $\kf$, we obtain

\begin{equation}
\kf=\hat{k}(n,\tau)e^{-\left(n^2-\frac{p+2}{p}\right)\int_{\tau}^t\hat{k}(0,\sigma)^{p+1}\,d\sigma}+\int_{\tau}^t h(s)e^{-\left(n^2-\frac{p+2}{p}\right)\int_{s}^t\hat{k}(0,\sigma)^{p+1}\,d\sigma}\,ds
    \label{eq:intfour}
\end{equation}

where $h(t)$ is given by:

\[h(t)=\sum_{\substack{\mathbf{q}\in\mathcal{A}_n\cap \Zel^{p+2}}}H\left(p,q_1,q_2\right)\Phi(\mathbf{q},t)\kst{p+2}\left(\mathbf{q},t\right),\]

Applying the \emph{Trapping Lemma}, we get

\[\left|h(t)\right|\leq C_p\kfc^{p}.\]

Also, from the Trapping Lemma, there exists $C,\mu>0,$ such that 
		
		\begin{equation}
		\left|\kf\right|\leq Ce^{-\left|\mu\right|},\quad\text{for}\quad t\geq \frac{T}{2},
		\label{eq:expbound}
		\end{equation}

We shall use these observations in proving the following decay (in time) estimates for the Fourier wave numbers of $k$.

	\begin{prop}
	\label{propfirstbound}
		There exists $\epsilon_0>0$ which depends on $p$, and a $\mu>0$ that depends also on $p$ and on $\psi$, such that if $t>\frac{T}{2}>0$ then there is a constant $b>0$ such that for any $0<\epsilon<\epsilon_0$, for $n\neq 0,\pm 1$, the following estimate holds for the solution of (\ref{eq:bvp}),
		
		\[\left|\kf\right|<be^{-\mu \left|n\right|}(T-t)^{\epsilon}\quad{whenever}\quad t>\frac{T}{2}.\] 
		
	\end{prop}
	
	\begin{proof} (See also the proof of Lemma 3.6 in \cite{CortissozMurcia})
		First we have
		
		\begin{eqnarray*}
			\left|\kf\right|&\leq&\left|\hat{k}(n,T-\delta)\right|e^{-\left(n^2-\frac{p+2}{p}\right)\int_{T-\delta}^t\hat{k}(0,s)^{p+1}\,ds}\\
		&&\qquad\qquad\qquad	+\int_{T-\delta}^{t}\left|h(s)\right|e^{-\left(n^2-\frac{p+2}{p}\right)\int_{s}^t\hat{k}(0,\sigma)^{p+1}\,d\sigma}\,ds\\
			&\leq& \left|\hat{k}(n,T-\delta)\right|\left(\frac{T-t}{\delta}\right)^{\eta\alpha(n,p)}\\
			&&\qquad\qquad\qquad+\int_{T-\delta}^{t}\left|h(s)\right|e^{-\left(n^2-\frac{p+2}{p}\right)\int_{s}^t\hat{k}(0,\sigma)^{p+1}\,d\sigma}\,ds\\
		\end{eqnarray*}
		
		where 
		
		\[\alpha(n,p)=\left(n^2-\frac{p+2}{p}\right)\frac{p}{p+1}.\]
		
		We are going to estimate the term inside the integral in the last inequality. As before we split $h(s)$ into sums of the form 
		
		\[J_{i_1,i_2,\dots,i_l}=\sum_{\substack{\mathbf{q}\in\mathcal{A}_0\cap \Zel^{p+2};\\q_j=0\Leftrightarrow j=i_1,i_2,\dots,i_l }}H\left(p,q_1,q_2\right)\kst{p+2}\left(\mathbf{q},s\right)\]

		using this, the Trapping Lemma and the observations after its statement, we get 
		
		\[\left|J_{i_1,i_2,\dots,i_l}\right|\leq \frac{Ce^{-\mu|n|}}{(T-s)^{\frac{l}{p+1}}},\]
		
		and since $l\leq p$, we finally obtain 
		
		\[\left|h(s)\right|\leq \frac{Ce^{-\mu|n|}}{(T-s)^{\frac{p}{p+1}}}\]
		
		Then we have
		
		\begin{align*}
			\left|\kf\right|&\leq \left|\hat{k}(n,T-\delta)\right|\left(\frac{T-t}{\delta}\right)^{\eta\alpha(n,p)}\\&+C\left(T-t\right)^{\eta\alpha(n,p)}e^{-\mu|n|}\int_{T-\delta}^{t}\frac{1}{\left(T-s\right)^{\eta\alpha(n,p)+\frac{p}{p+1}}}\,ds\\
		\end{align*}
		
		Without lost of generality we can assume $\eta\leq\frac{1}{2}$. Now using again (\ref{eq:expbound}) and the fact that $\alpha(n,p)\geq \alpha(2,p)>0$, obtain
		
		\begin{align*}
		\left|\kf\right|&\leq \left|\hat{k}\left(n,T-\delta\right)\right|\left(\frac{T-t}{\delta}\right)^{\eta\alpha(n,p)} +Ce^{-\mu |n|}(T-t)^{1-\frac{p}{p+1}}\\
		&\leq \frac{b}{2}e^{-\mu|n|}\left(\left(T-t\right)^{\eta\alpha(n,p)}+\left(T-t\right)^{1-\frac{p}{p+1}}\right),	\end{align*}
		
		then we have, 
		
		\[\left|\kf\right|\leq b e^{-\mu|n|}(T-t)^{\epsilon},\]
		
		for any $0<\epsilon<\min\left\{\frac{1}{2}\alpha\left(2,p\right),1-\frac{p}{p+1}\right\}=\epsilon_0$.
	\end{proof}
	
	Next we are going to improve on the decay estimates of the Fourier coefficients. 
	To be able to do this, we will need the following
	lemma.
	
	\begin{lema}
	\label{lem:blowup0}
	There is a $t_0$ such that if $t\in (t_0,T)$, then we have the estimates 
	
	\[
	\left(\frac{p+1}{p}\right)^{\frac{1}{p+1}}
	\kfc\leq \frac{1}{\left[\left(T-t\right)-c_1\left(T-t\right)^{1+\frac{2}{p+1}}\right]^{\frac{1}{p+1}}},\]
	and
	\[
	\left(\frac{p+1}{p}\right)^{\frac{1}{p+1}}
	\kfc\geq
	\frac{1}{\left[\left(T-t\right)+c_1\left(T-t\right)^{1+\frac{2}{p+1}}\right]^{\frac{1}{p+1}}}.
	\]
\end{lema}
\begin{proof}
	We have that the following differential inequality
	
	\[\dif{}\kfc\leq \frac{1}{p}\kfc^{p+2}+A\kfc^p,\]
	
	holds for a constant $A>0$ independent of $t$. This is equivalent to
	
	\[\frac{1}{\kfc^{p+2}}\dif{}\kfc\leq \frac{1}{p}+A\kfc^{-2}.\]
	
	Using Lemma 5, from the previous differential inequality we obtain
	
	\[\frac{1}{\kfc^{p+2}}\dif{}\kfc\leq \frac{1}{p}+C\left(T-t\right)^{\frac{2}{p+1}},\]
	
	The result follows by integration. For the other inequality, notice that there exists a constant $A'$ so we 
	also have the following differential inequality
	
	\[\dif{}\kfc\geq \frac{1}{p}\kfc^{p+2}+A'\kfc^p.\]
	
	\end{proof}

	In order to proceed we use previous proposition to estimate the integral
	
	\[I=\frac{p+1}{p}\int_{T-\delta}^t\hat{k}(0,\tau)^{p+1}\,d\tau\]
	
	from below, from previous lemma, since $\delta$ is small, using Taylor's Theorem, this shows
	
	\[I\geq \int_{T-\delta}^{t}\frac{d\tau}{T-\tau+(T-\tau)^{1+\frac{2}{p+1}}}\geq -\ln\left(\frac{T-t}{\delta}\right)-c,\]
	
	$c>0,$ and using this and (\ref{eq:intfour}), we get 
	
	\begin{eqnarray}
	\left|\kf\right|&\leq& C\left|\hat{k}(n,T-\delta)\right|\left(\frac{T-t}{\delta}\right)^{\alpha(n,p)}\\ \notag
	&&+C(T-t)^{\alpha(n,p)}\int_{T-\delta}^{t}\left(\frac{1}{T-s}\right)^{\alpha(n,p)}h(s)\,ds
	\label{eq:ineqimp}
	\end{eqnarray}
	 here
	 
	 \begin{equation}
	 \label{definitionofalpha}
	 \alpha(n,p)=(n^2-\frac{p+2}{p})\frac{p}{p+1}.
	 \end{equation}
	 
	 Using the estimate of the proposition and the fact that if $\mathbf{q}\in\mathcal{A}_{n}$ then $\mathbf{q}$ has at least two entries different from $0$,we get
	 
	 \[\left|h(s)\right|\leq\frac{Ce^{-\mu'|n|}}{\left(T-s\right)^{\frac{p}{p+1}}}(T-s)^{2\epsilon}.\]
	 
	 If we introduce the bound from Proposition \ref{propfirstbound} in (\ref{eq:ineqimp}) we get,
	 
	 \begin{eqnarray*}
	 \left|\kf\right|&\leq& C\left|\hat{k}(n,T-\delta)\right|\left(\frac{T-t}{\delta}\right)^{\alpha(n,p)}\\
	 &&+C(T-t)^{\alpha(n,p)}e^{-\mu'|n|}\int_{T-\delta}^{t}\frac{(T-t)^{2\epsilon}}{\left(T-s\right)^{\alpha(n,p)+\frac{p}{p+1}}}\,ds
	 \end{eqnarray*}
	 
	 for a well chosen $0<\mu'<\mu$, and from which we obtain the estimate (we use again the fact $\alpha(n,p)\geq \alpha(2,p)$)
	 
	 \[\left|\kf\right|\leq C'e^{-\mu'|n|}(T-t)^{\min\{\alpha(2,p),1-\frac{p}{p+1}+2\epsilon\}}.\]
	 
	 If we assume that $\alpha(2,p)>1-\frac{p}{p+1}+2\epsilon,$ using this new bound an pluggin it into (\ref{eq:ineqimp}), we improve again our estimate on $\kf$:
	 
	 \[\left|\kf\right|\leq C''e^{-\mu''|n|}(T-t)^{\min\left\{\alpha(2,p),3\left(1-\frac{p}{p+1}\right)+4\epsilon\right\}}\quad\left(0<\mu''<\mu'\right)\],
	 
	 Finally , if we repeat this procedure a finite number of times we arrive at  
	 
	 \begin{equation}
	 \label{decayintimeestimate}
	 \left|\kf\right|\leq De^{-\xi|n|}(T-t)^{\alpha(2,p)},\quad n \neq 0, \pm 1.
	 \end{equation}
	 where $\xi$ is a constant independent of $n$, and $\alpha\left(2,p\right)$ is defined
	 by (\ref{definitionofalpha}) (so is value its $\left(3p-2\right)/\left(p+1\right)$). 
	 
	 Now, we must show now that the wavenumbers $\hat{k}\left(\pm 1,t\right)$
	 satisfy the same estimate.
	 In this case we write
	 \[
	 z\left(n,t\right)=\left(\frac{p+1}{p}\right)^{\frac{1}{p+1}}
	 \left(T-t\right)^{\frac{1}{p+1}}
	 \hat{k}\left(n,t\right).
	 \]
	 And we have an identity which follows from $Q\left(k\right)=0$ (here we use that $\hat{z}\left(-1,t\right)$ is the conjugate
	 of $\hat{z}\left(1,t\right)$, as $z$ is real valued)
	 \[
	 \sum_{n=0}^{\infty}
	 \binom{2n}{n}\left|z\left(1,t\right)\right|^{2n} z\left(1,t\right)
	 = \sum_{m=2}^{\infty}\sum'_{q_1+q_2+\dots+q_m=1}z\left(q_1\right)
	 \cdots
	 z\left(q_m,t\right),
	 \]
	 where the prime ($'$) in the inner sum 
	 of the righthand side indicates that at least one of the
	 $q_j\neq \pm 1$. Using similar computations as in the proof of
	 Lemma \ref{controllingsmallwavenumbers}, together with (\ref{decayintimeestimate}), we can conclude that
	 \[
	  \sum_{n=0}^{\infty}\binom{2n}{n}\left|z\left(1,t\right)\right|^{2n} z\left(1,t\right)=O\left(\left(T-t\right)^{\frac{3p-1}{p+1}}\right),
	 \]
	 hence, if $\delta>0$ is small enough, we can deduce that
	 \[
	  z\left(1,t\right)=O\left(\left(T-t\right)^{\frac{3p-1}{p+1}}\right),
	 \]
	 which is just that $\left|\hat{k}\left(\pm 1,t\right)\right|\leq C\left(T-t\right)^{\frac{3p-2}{p+1}}$.
	 So we have proved
	\begin{prop}
	\label{uniformizationunnormalised}
	Let $\psi>0$ is a smooth $2\pi$-periodic function 
	which satisfies that $Q\left(\psi\right)=0$. There exists a positive constant $c_p$ such that if
	
	\[\hat{\psi}(0)\geq c_p\norma{\psi}{2},\]
	
	then a solution to (\ref{eq:bvp}) satisfies
	
	\[\norma{k\left(\theta,t\right)-\hat{k}\left(0,t\right)}{C^{k}[0,2 \pi]}\leq M_{p,k}(T-t)^{\frac{3p-2}{p+1}}\]
	
	where $T$ is the blow-up time and $M_{p,k}$ is a constant that depends only on $p,k$ and $\psi$.
	\end{prop}
	
 	We normalize the solution of (\ref{eq:bvp}) by means of the following transformation:
	
	\[\tilde{k}\left(\theta,t\right)=\left(\frac{p+1}{p}\right)^{\frac{1}{p+1}}\left(T-t\right)^{\frac{1}{p+1}}k\left(\theta,t\right),\quad \tau=-\frac{1}{p+1}\log\left(1-\frac{t}{T}\right)\]
	
	Applying chain rule, we obtain the following normalized version of (\ref{eq:bvp})
	
	\begin{equation}
	\begin{cases}
	\dfrac{\partial\tilde{k}}{\partial \tau}=p\tilde{k}^{p+1}\dfrac{\partial^2\tilde{k}}{\partial\theta^2}+p(p-1)\tilde{k}^p
	\left(\dfrac{\partial\tilde{k}}{\partial\theta}\right)^2+\tilde{k}^{p+2}-\tilde{k}\quad\text{ in }\quad\left[0,2\pi\right]\times\left(0,\infty\right)\\
	\tilde{k}(\theta,0)=\left(\dfrac{(p+1)T}{p}\right)^{\frac{1}{p+1}}\psi\left(\theta\right).
	\end{cases}
	\label{eq:bvpnormalized}
	\end{equation}	
	
	Using this normalisation, Proposition \ref{uniformizationunnormalised} translates into: 
	\begin{cor}
	\label{firstuniformization}
Let $\psi>0$ is a smooth $2\pi$-periodic function which satisfies that $Q\left(\psi\right)=0$. There exists a positive constant $c_p$ such that if
	
	\[\hat{\psi}(0)\geq c_p\norma{\psi}{2},\]
	
	then the normalization $\tilde{k}$ of $k$ satisfies:
	
	\[\norma{\tilde{k}-\hat{\tilde{k}}(0,t)}{C^{k}[0,2 \pi]}\leq M_{p,k}e^{-(3p-1)\tau},\]	
	where $M_{p,k}$ is a positive constant that depends only on $p,k$ and $\psi$. \label{cor:impor2}
	\end{cor}

	\section{Exponential convergence of the normalised curvature towards 1: Proof of the main result}
	
	We will need the following improvement over Lemma \ref{lem:blowup0}.
	\begin{lema}
	The following estimates hold
	\[
	\left(\frac{p+1}{p}\right)^{\frac{1}{p+1}}\hat{k}\left(0,t\right)
	\leq \frac{1}{\left[\left(T-t\right)-a_0\left(T-t\right)^{1+\frac{6p-2}{p+1}}\right]^{\frac{1}{p+1}}},
	\]
	and
	\[
	\left(\frac{p+1}{p}\right)^{\frac{1}{p+1}}\hat{k}\left(0,t\right)\geq
	\frac{1}{\left[\left(T-t\right)+a_1\left(T-t\right)^{1+\frac{6p-2}{p+1}}\right]^{\frac{1}{p+1}}}.
	\]
	\end{lema}
	
	\begin{proof}
     Notice that using (\ref{decayintimeestimate}) and the equation satisfied by $\kfc$, we have the differential inequality,
     which is valid for a constant $A>0$ 
	
	\[\dif{}\kfc\leq \frac{1}{p}\kfc^{p+2}+A\left(T-t\right)^{\frac{6p-4}{p+1}}\kfc^{p}.\]
	
	Integrating, from $t$ to $T$, we obtain the second inequality. Analogously for a constant $A'$, we have the differential inequality
		\[\dif{}\kfc\geq \frac{1}{p}\kfc^{p+2}-A'\left(T-t\right)^{\frac{6p-4}{p+1}}\kfc^{p},\]
	which by integration gives the first inequality.	
\end{proof} 
	Finally we have our main result.
	
	\begin{theorem}
		Let $\psi>0$ be the initial condition  of (\ref{eq:bvp}) (so it is the curvature
		function of a convex simple curve). Then there exists a constant $c_p>0$ such that if
	
		\[\hat{\psi}(0)\geq c_p\norma{\psi}{2}, \] 
		
		then the solution to (\ref{eq:bvpnormalized}) satisfies 
		
		\[\norma{\tilde{k}-1}{C^{l}\left[0,2\pi\right]}\leq C_{p,l}e^{-\left(3p-1\right)\tau},\]
		
		where $C_{p,l}$ is a constant that only depends on the initial condition $\psi$ and $p$ and $l$. 
		\end{theorem}
		\begin{proof}
			Let 
			
			\[\tilde{u}(0,t)=\frac{1}{2\pi}\int_{0}^{2\pi}\tilde{k}\left(\theta,t\right)\,d\theta=\left(\frac{p+1}{p}\right)^{\frac{1}{p+1}}\left(T-t\right)^{\frac{1}{p+1}}\kfc,\]
			
			now we compute 
			
			\begin{align*}
			\tilde{u}(0,t)-1&= \left(\frac{p+1}{p}\right)^{\frac{1}{p+1}}\kfc(T-t)^{\frac{1}{p+1}}-1\\
			&\leq \frac{(T-t)^{\frac{1}{p+1}}}{\left[(T-t)-c_1(T-t)^{1+\frac{6p-2}{p+1}}\right]^{\frac{1}{p+1}}}-1
			\leq C(T-t)^{\frac{6p-2}{p+1}},
			\end{align*}
			
			Analogously, 
	
	\[1-\tilde{u}(0,t)\leq C(T-t)^{\frac{6p-2}{p+1}},\]
	
	In this case, $ \displaystyle{e^{-\tau}=\left(\frac{T-t}{T}\right)^{\frac{1}{p+1}}}$, then
	
	\[\left|\frac{1}{2\pi}\int_{0}^{2\pi}\tilde{k}(\theta,\tau)\,d\theta-1\right|\leq Ce^{-(6p-2)\tau}\]
	
	Applying the triangular inequality and Corollary \ref{cor:impor2} we can conclude that
	
	\[\norma{\tilde{k}-1}{C^{l}\left[0,2\pi\right]}\leq\norma{\tilde{k}-\tilde{u}}{C^{l}\left[0,2\pi\right]}+\norma{\tilde{u}-1}{C^{l}\left[0,2\pi\right]}\leq C_{p,l}e^{-(3p-1)\tau,} \]
	
	for some constant $C_{p,l}$ that depends on $p$ and $l$.

\end{proof}

\subsection{Final Remarks}

The rate of convergence obatined in Theorem, seems to be the best possible in general. We have not been able to
produce an example where the rate given in the theorem is met; however, to justify our claim, we refer to the comments
after the statement of Theorem 2.2 in \cite{CortissozMurcia}: The first positive eigenvalue of the elliptic part
of (\ref{eq:bvpnormalized}), i.e., the left hand side of the equation, when linearised around the steady solution 
$\tilde{k}\equiv 1$ is precisely $3p-1$.


\begin{thebibliography}{99}

\bibitem{Andrews}
B. Andrews, Evolving convex curves, Calc. Var. Partial Differential Equations 7 (1998) 315--371.

\bibitem{AndrewsBryan}
 B. Andrews and P. Bryan, Curvature bound for curve shortening flow via distance comparison and a direct proof of Grayson's theorem. J. Reine Angew. Math. 653 (2011), 179--187.

\bibitem{ChouZhu}
 K.-S. Chou and X.-P. Zhu, The curve shortening problem. Chapman and Hall/CRC, Boca Raton, FL, 2001. x+255 pp.

\bibitem{CoddingtonLevinson}
 E. A. Coddington and N. Levinson, Theory of ordinary differential equations. McGraw-Hill Book Company, Inc.,
 New York-Toronto-London, 1955. xii+429 pp.

\bibitem
{CortissozMurcia}
 J. C. Cortissoz and A. Murcia, On the stability of m-fold circles and the dynamics of generalized curve shortening flows. 
 J. Math. Anal. Appl. 402 (2013), no. 1, 57–70.
 
 \bibitem{GageHamilton}M. Gage and R.S. Hamilton, The heat equation shrinking convex plane curves.
J. Differential Geom. {\bf 23} (1986), no. 1, 69--96.

\bibitem{Huang} R. L. Huang, Blow-up rates for the general curve shortening flow. 
J. Math. Anal. Appl. {\bf 383} (2011), no 2, 482--487.

\bibitem{Murcia}
A. Murcia, The Ricci Flow On Surfaces with boundary and Quasilinear Evolution Equations in $\S^1$. PhD thesis,
Universidad de los Andes, 2015.

\bibitem{Sesum}
 N. Sesum, Rate of convergence of the mean curvature flow. Comm. Pure Appl. Math. 61 (2008), no. 4, 464--485.

\end{thebibliography}
\end{document}